\documentclass[11pt,a4paper]{article}
\usepackage{indentfirst}
\usepackage{amsfonts}
\usepackage{amssymb}
\usepackage{mathrsfs}
\usepackage{amsmath}
\usepackage{amsthm}
\usepackage{enumerate}
\usepackage{cite}
\usepackage{mathrsfs}
\allowdisplaybreaks
\usepackage{geometry}
\usepackage{ifpdf}
\ifpdf
\usepackage[colorlinks=true,linkcolor=blue,citecolor=red, final,backref=page,hyperindex]{hyperref}
\else
\usepackage[colorlinks,final,backref=page,hyperindex,hypertex]{hyperref}
\fi

\usepackage{txfonts}
\usepackage{indentfirst}
\usepackage{latexsym}
\usepackage[all]{xy}
\def\<{\langle}
\def\>{\rangle}
\def\a{\alpha}
\def\b{\beta}

\newtheorem{thm}{Theorem}[section]

\newtheorem{prop}[thm]{Proposition}
\newtheorem{ex}[thm]{Example}

\theoremstyle{definition}
\newtheorem{defn}{Definition}[section]

\theoremstyle{remark}
\newtheorem{re}{Remark}[section]
\begin{document}
	\title{\bf Cohomology of compatible BiHom-Lie algebras}
		\author{\bf  Asif Sania, Basdouri Imed, Bouzid Mosbahi, Nacib Saber }
	\author{{ Asif Sania \footnote{Corresponding author,  E-mail: 11835037@zju.edu.cn}, Basdouri Imed $^{1}$
			\footnote { Corresponding author,  E-mail:  basdourimed@yahoo.fr}
            ,\ Bouzid Mosbahi$^{2}$
			\footnote { Corresponding author,  E-mail: mosbahibouzid@yahoo.fr}
			,\  Nacib Saber $^{3}$
			\footnote { Corresponding author,  E-mail: sabernacib83@hotmail.com}
		}\\
		{\small 1. Nanjing University of Information Science and Technology, Nanjing, PR China.}\\
			{\small 2.  University of Gafsa, Faculty of Sciences Gafsa, 2112 Gafsa, Tunisia } \\
        {\small 3.  University of Sfax, Faculty of Sciences Sfax,  BP
			1171, 3038 Sfax, Tunisia}\\
		{\small 4.  University of Sfax, Faculty of Sciences Sfax,  BP
			1171, 3038 Sfax, Tunisia}}
	\date{}
	\maketitle
%%%%%%%%%%%%%%%%%%%%%%%%%%%%%%%%%%%%%%%%%%%%%%%%%%%%%%%%%%%%%%%%%%%%%%%%%%%%%%%%%%%%%%%%%%%%%%%%%%%%%%%%%%
%%%%%%%%%%%%%%%%%%%%%%%%%%%%%%%%%%%%%%%%%%%%%%%%%%%%%%%%%%%%%%%%%%%%%%%%%%%%%%%%%%%%%%%%%%%%%%%%%%%%%%%%%%
%%%%%%%%%%%%%%%%%%%%%%%%%%%%%%%%%%%%%%%%%%%%%%%%%%%%%%%%%%%%%%%%%%%%%%%%%%%%%%%%%%%%%%%%%%%%%%%%%%%%%%%%%%
%%%%%%%%%%%%%%%%%%%%%%%%%%%%%%%%%%%%%%%%%%%%%%%%%%%%%%%%%%%%%%%%%%%%%%%%%%%%%%%%%%%%%%%%%%%%%%%%%%%%%%%%%%

	\begin{abstract}
This paper defines compatible BiHom-Lie algebras by twisting the compatible Lie algebras by two linear commuting maps. We show the characterization of compatible BiHom-Lie algebra as a Maurer-Cartan element in a suitable bidifferential graded Lie algebra. We also define a cohomology theory for compatible BiHom-Lie algebras.
	\end{abstract}
%%%%%%%%%%%%%%%%%%%%%%%%%%%%%%%%%%%%%%%%%%%%%%%%%%%%%%%%%%%%%%%%%%%%%%%%%%%%%%%%%%%%%%%%%%%%%%%%%%%%%%%%%%
%%%%%%%%%%%%%%%%%%%%%%%%%%%%%%%%%%%%%%%%%%%%%%%%%%%%%%%%%%%%%%%%%%%%%%%%%%%%%%%%%%%%%%%%%%%%%%%%%%%%%%%%%%
%%%%%%%%%%%%%%%%%%%%%%%%%%%%%%%%%%%%%%%%%%%%%%%%%%%%%%%%%%%%%%%%%%%%%%%%%%%%%%%%%%%%%%%%%%%%%%%%%%%%%%%%%%
%%%%%%%%%%%%%%%%%%%%%%%%%%%%%%%%%%%%%%%%%%%%%%%%%%%%%%%%%%%%%%%%%%%%%%%%%%%%%%%%%%%%%%%%%%%%%%%%%%%%%%%%%%	
	
	
	\numberwithin{equation}{section} 	\tableofcontents

%%%%%%%%%%%%%%%%%%%%%%%%%%%%%%%%%%%%%%%%%%%%%%%%%%%%%%%%%%%%%%%%%%%%%%%%%%%%%%%%%%%%%%%%%%%%%%%%%%%%%%%%%%
%%%%%%%%%%%%%%%%%%%%%%%%%%%%%%%%%%%%%%%%%%%%%%%%%%%%%%%%%%%%%%%%%%%%%%%%%%%%%%%%%%%%%%%%%%%%%%%%%%%%%%%%%%
%%%%%%%%%%%%%%%%%%%%%%%%%%%%%%%%%%%%%%%%%%%%%%%%%%%%%%%%%%%%%%%%%%%%%%%%%%%%%%%%%%%%%%%%%%%%%%%%%%%%%%%%%%
%%%%%%%%%%%%%%%%%%%%%%%%%%%%%%%%%%%%%%%%%%%%%%%%%%%%%%%%%%%%%%%%%%%%%%%%%%%%%%%%%%%%%%%%%%%%%%%%%%%%%%%%%%
	\section{Introduction}
The motivation to study BiHom-Lie algebra structure lies in its importance in the physics and deformations theory  of Lie algebras, in particular Lie algebras of vector fields. BiHom-Lie algebra was first introduced by
Cheng and Qi in \cite{2}, is a $4$-tuple $(g, [\cdot,\cdot], \alpha, \beta)$, where $g$ is a $\mathbb{K}$-linear space, $\alpha, \beta: g \rightarrow g,$  are linear maps and $[\cdot,\cdot] : g \otimes g \longrightarrow g$ is Lie bracket operation on $g$, with notation
$[\cdot,\cdot](p \otimes q) = [p,q]$, satisfying the following conditions, for all $p, q, r \in g:$
\begin{center}
 $\alpha \circ \beta = \beta \circ \alpha,$\\
 $[\beta(p), \alpha(q)] = -[\beta(q),\alpha(p)]$ (skew-symmetry),\\
 $[\beta^{2}(p),[\beta(q),\alpha(r)]] + [\beta^{2}(q),[\beta(r),\alpha(p)]] + [\beta^{2}(r),[\beta(p),\alpha(q)]]=0$  (BiHom-Jacobi condition).
\end{center}
The primary characteristic of BiHom-Lie algebras is that they are an extension of Hom-Lie algebra with one twist map $\alpha$, which is defined in \cite{4, 10}, where the identities characterizing BiHom-Lie algebras are twisted by two twist maps $\alpha, \beta$. If we consider $\alpha=\beta$ then the theory of BiHom-Lie algebras deform to Hom-Lie algebras and by putting $\alpha=\beta=id$, then we get a Lie algebra structure.

It's important to note that (linearly) compatible products have gained a significant amount of attention. In this respect, two products of a particular kind defined on the same vector space are said to be compatible if their sum also defines the same type of algebraic structure. They appeared in a variety of mathematical physics and mathematics problems. For instance, the idea of two Hom-Lie structures being compatible first arose in \cite{3}, where the author developed a generalized algebraic structure with a single commuting multiplicative linear map. This structure was referred to as compatible Hom-Lie algebra. Das recently developed a cohomology theory for the compatible Hom-Lie algebras. This cohomology is based on the characterization of a compatible Hom-Lie algebra as a Maurer-Cartan element in a bidifferential graded Lie algebra. By realizing the importance of both the BiHom Lie algebras and compatible Lie algebras, we study the compatible BiHom-Lie algebras and provide its cohomology theory. Same as Bihom Lie algebras, compatible BiHom-Lie algebras also return to compatible Hom-Lie algebras by choosing the same twist map.

The paper is organized as follows. In Section 2 (preliminary section), we recall some important definitions and notions of BiHom-Lie algebras and bidifferential graded Lie algebras. Section 3 introduces compatible BiHom-Lie algebras and their Maurer-Cartan characterizations in a suitably constructed bidifferential graded Lie algebra. We also define the notion of representation of a compatible BiHom-Lie algebra and construct the semidirect product.
The cohomology of a compatible BiHom-Lie algebra with coefficients in a representation is given in Section 4.
%%%%%%%%%%%%%%%%%%%%%%%%%%%%%%%%%%%%%%%%%%%%%%%%%%%%%%%%%%%%%%%%%%%%%%%%%%%%%%%%%%%%%%%%%%%%%%%%%%%%%%%%%%
%%%%%%%%%%%%%%%%%%%%%%%%%%%%%%%%%%%%%%%%%%%%%%%%%%%%%%%%%%%%%%%%%%%%%%%%%%%%%%%%%%%%%%%%%%%%%%%%%%%%%%%%%%
%%%%%%%%%%%%%%%%%%%%%%%%%%%%%%%%%%%%%%%%%%%%%%%%%%%%%%%%%%%%%%%%%%%%%%%%%%%%%%%%%%%%%%%%%%%%%%%%%%%%%%%%%%
%%%%%%%%%%%%%%%%%%%%%%%%%%%%%%%%%%%%%%%%%%%%%%%%%%%%%%%%%%%%%%%%%%%%%%%%%%%%%%%%%%%%%%%%%%%%%%%%%%%%%%%%%%	
	\section{BiHom-Lie algebras and bidifferential graded Lie algebras }
\subsection{BiHom-Lie algebras}	 We start by recalling some important remarks about BiHom-Lie algebras defined in the introductory Section.
%\begin{defn}A BiHom-Lie algebra over a field $\mathbb{K}$ is a $4$-tuple $(g, [\cdot,\cdot], \alpha, \beta)$, where $g$ is a $\mathbb{K}$-linear space, $\alpha, \beta: g \rightarrow g,$  are linear maps and $[\cdot,\cdot] : g \otimes g \longrightarrow g$ is Lie bracket operation on $g$, with notation$[\cdot,\cdot](p \otimes q) = [p,q]$, satisfying the following conditions, for all $p, q, r \in g:$\begin{center} $\alpha \circ \beta = \beta \circ \alpha,$\\$[\beta(p), \alpha(q)] = -[\beta(q),\alpha(p)]$ (skew-symmetry),\\$[\beta^{2}(p),[\beta(q),\alpha(r)]] + [\beta^{2}(q),[\beta(r),\alpha(p)]] + [\beta^{2}(r),[\beta(p),\alpha(q)]]=0$ (BiHom-Jacobi condition)\end{center}\end{defn}
\begin{re}
A BiHom-Lie algebra is called a multiplicative BiHom-Lie algebra if $\alpha$ and $\beta$ are
algebra morphisms, i.e.  we have $\alpha([p,q]) = [\alpha(p), \alpha(q)], \beta([p,q]) = [\beta(p), \beta(q)]$, for any $p,q \in g$.
\end{re}
\begin{re}
A multiplicative BiHom-Lie algebra is called a regular BiHom-Lie algebra if $\alpha, \beta$ are
bijective maps.
\end{re}
\begin{re}
Obviously, a BiHom-Lie algebra $(g, [\cdot,\cdot], \alpha,\beta)$ for which $\alpha = \beta$ is just a Hom-Lie algebra $(g, [\cdot,\cdot], \alpha)$.
\end{re}
\begin{ex}
Let $(g, \mu, \alpha, \beta)$ be a BiHom-associative algebra with a bilinear map  $\mu: g\otimes g\to g$ defined by $\mu(p,q)=p\cdot q$  and bijective linear maps $\alpha, \beta:g\to g$. Define a new multiplication $[\cdot,\cdot]$ by $[p,q]= p.q -(\alpha^{-1} \beta(q)).(\alpha\beta^{-1}(p))$, for every $p, q \in g$. We say $(g, [\cdot,\cdot], \alpha, \beta)$ is a BiHom Lie algebra.
\end{ex}
\begin{ex}
	Let $(g,[\cdot,\cdot])$ be an ordinary Lie algebra over a field $\mathbb{K}$ and let $\alpha,\beta:g \rightarrow g$ two commuting linear maps such that $\alpha([p,q])=[\alpha(p),\alpha(q)]$ and $\beta([p,q])=[\beta(p),\beta(q)]$ for all $a, a^{\prime}\in g$. Define the linear maps $\{\cdot,\cdot\}:g\otimes g \longrightarrow g$ by
	$\{p,q\}= [\alpha(p),\beta(q)]$ for all $p,q \in g$. Then $g_{(\alpha,\beta)}=( g,\{\cdot,\cdot\},\alpha,\beta)$ is a BiHom-Lie algebra called the twist of Lie algebra $(g, [\cdot,\cdot])$.
\end{ex}
\begin{defn}
Let $(g, [\cdot,\cdot], \alpha,\beta)$ and $(g^{\prime}, [\cdot,\cdot]^{\prime}, \alpha^{\prime},\beta^{\prime})$ be two BiHom-Lie algebras.  A linear map $\phi : g \rightarrow g^{\prime}$ is said to be a BiHom-Lie algebra morphism if $\alpha^{\prime} \circ \phi = \phi \circ \alpha,\beta^{\prime} \circ \phi = \phi \circ \beta$ and $\phi[p, q] = [\phi(p), \phi(q)]^{\prime},$ for $p, q \in g$.
\end{defn}
\begin{ex} Given two BiHom-Lie algebras $(g, [\cdot,\cdot], \alpha,\beta)$ and $(g^{\prime}, [\cdot,\cdot]^{\prime}, \alpha^{\prime},\beta^{\prime})$, there is a BiHom-Lie algebra $(g\oplus g^{\prime}, [\cdot,\cdot ]_{g\oplus g^{\prime}}, \alpha+ \alpha^{\prime}, \beta+ \beta^{\prime})$, where the skew-symmetric bilinear map $[\cdot,\cdot]_{g\oplus g^{\prime}}: \wedge^{2}(g\oplus g^{\prime})\longrightarrow g\oplus g^{\prime}$ is given by $[(p_{1},q_{1}),(p_{2},q_{2})]_{g\oplus g^{\prime}}= ([p_{1},p_{2}],[q_{1},q_{2}]^{\prime})$ for all $p_{1}, p_{2} \in g, q_{1}, q_{2} \in g^{\prime}$ and the linear maps $(\alpha+ \alpha^{\prime}),(\beta+ \beta^{\prime}): g\oplus g^{\prime}\longrightarrow g\oplus g^{\prime}$ are given by
	$$(\alpha+\alpha^{\prime})(p, q)= (\alpha(p),\alpha^{\prime}(q)), ~~(\beta+ \beta^{\prime})(p, q)= (\beta(p),\beta^{\prime}(q)),~~ \forall p\in g, q\in g^{\prime}$$
\end{ex}
Next, we recall the graded Lie bracket (called the Nijenhuis-Richardson bracket) whose Maurer-Cartan
elements are given by the BiHom-Lie algebra structure. This generalizes the classical Nijenhuis-Richardson bracket in the context of Lie algebra. Let $g$ be a vector space and $\alpha, \beta: g \rightarrow g$  be two linear maps. For each $n \geq 0$, consider the spaces $$C^{0}_{BiHom}(g, g)= \{p \in g ~|~ \alpha(p)= p\textit{ and }\beta(p)= p\}$$ and $$C^{n}_{BiHom}(g, g) = \{f : \otimes^{n}g \rightarrow g~|~\alpha \circ f = f \circ \alpha^{\otimes^{n}}, \beta \circ f = f \circ \beta^{\otimes ^{n}}\},~~for ~~n \geq 1. $$  Then the shifted graded vector space
$C^{\ast+ 1}_{BiHom}(g, g)= \oplus_{n\geq0}C^{n+ 1}_{BiHom}(g, g)$ carries a graded Lie bracket defined as follows:\\
For $P \in C^{m+1}_{BiHom} (g, g)$ and $Q \in C^{n+1}
_{BiHom}(g,g)$, the Nijenhuis-Richardson bracket
$[P,Q]_{NR} \in C^{m+n+1}_{BiHom} (g, g)$ given by
$$[P,Q]_{NR} = P \diamond Q - (-1)^{mn} Q \diamond P,$$ where $$(P \diamond Q)(p_{1},..., p_{m+n+1}) = \sum_{\sigma \in Sh(n+1,m)} (-1)^{\sigma} P(Q(p_{\sigma(1)},..., p_{\sigma(n+1)}), \alpha\beta^{n}(p_{\sigma(n+2)}),..., \alpha\beta^{n}(p_{\sigma(n+m+1)})).$$
With this notation, we have the following.
\begin{prop}
	Let $g$ be a vector space with $\alpha,\beta: g \rightarrow g  $ being two linear maps.
	Then the BiHom-Lie bracket on $g$ is precisely the Maurer-Cartan element  in the graded Lie algebra $(C^{\ast+1}_{BiHom(g, g)}, [\cdot,\cdot]_{NR})$.
\end{prop}In the following, we recall the Chevalley-Eilenberg cohomology of a BiHom-Lie algebra $(g, [. , .], \alpha, \beta)$ with
coefficients in a representation.
\begin{defn}
	Let $(g, [\cdot, \cdot], \alpha, \beta)$ be a BiHom-Lie algebra. A representation of $g$ is a
	$4$-tuple $(V, \bullet, \alpha_{V}, \beta_{V})$, where $V$ is a linear space, $\alpha_{V}, \beta_{V} : V \rightarrow V$ are two commuting
	linear maps and representation $\bullet: g \otimes V \longrightarrow V$ defined by $(p, v)\longmapsto p \bullet v$ 	is a bilinear operation (called the action) satisfying:
\begin{enumerate}
	\item $\alpha(p) \bullet \alpha_{V}(v)=\alpha_{V}(p \bullet v)$,
	\item$\beta(p) \bullet \beta_{V}(v)=\beta_{V}(p \bullet v)$,
	\item$[\beta(p),q]\bullet \beta_{V}(v)= \alpha\beta(p)\bullet(q\bullet v)- \beta(y)\bullet(\alpha(p)\bullet v)$,
\end{enumerate} for all $p, q \in g$ and $v\in V$.
\end{defn}
\begin{ex}
	Let $(g, [\cdot, \cdot], \alpha, \beta)$ be a BiHom-Lie algebra and $(V, \bullet, \alpha_{V}, \beta_{V})$ be a representation of $g$. Assume that the maps $\alpha, \beta$ and $\alpha_{V}, \beta_{V}$ are bijective. Then the semi direct products $g\ltimes V= (g\oplus V, [.,.], \alpha\oplus\alpha_{V}, \beta\oplus\beta_{V})$ is a BiHom-Lie algebra. Where $\alpha\oplus\alpha_{V},\beta\oplus\beta_{V}:g\oplus V \longrightarrow g\oplus V$ are defined by
	$$(\alpha\oplus\alpha_{V})(p,a)=(\alpha(p),\alpha_{V}(a)), (\beta\oplus\beta_{V})(p,a)=(\beta(x),\beta_{V}(a))$$ and the bracket $[\cdot, \cdot]$ is defined by $$[(p,a),(q,b)]=([p,q], p\bullet b-  \alpha^{-1}\beta(y) \bullet  \alpha_{V}\beta^{-1}_{V}(a))$$
	for all $p, q \in g$ and $a,b \in V$.
\end{ex}It follows that any BiHom-Lie algebra $(g, [\cdot,\cdot], \alpha,\beta)$ is a representation of itself with the action given by the bracket $[. , .]$. This is called the adjoint representation. Let $(g, [. , .], \alpha,\beta)$ be a BiHom-Lie algebra and $(V,\bullet,\alpha_{V},\beta_{V})$ be a representation on it. For each $n \geq 0$, we define the $n$-th cochain on $g$ with the coefficients in the representation $(V, \bullet, \alpha,\beta)$  which is the set of skew symmetric $n$-linear maps from $g\otimes g \otimes \cdots\otimes g$ ($n$-times) to $V,$ denote by $C^{n}(g, V)$. More specifically $$C^{n}(g, V ) = \{f ~|~f : \otimes^{n}g \longrightarrow V ~\textit{ is a multilinear map}\}.$$
A $n$-BiHom-cochain on $g$ with the coefficients in $V$ is defined to be a $n$-cochain $f \in C^{n}(g,V)$ such that is compatible with $\alpha,\beta$ and $\alpha_{V}, \beta_{V}$ in the sense that $\alpha_{V}\circ f= f\circ\alpha$ and $\beta_{V}\circ f=f\circ\beta$ i.e$$\alpha_{V}(f(p_{1},...,p_{n}))= f(\alpha(p_{1}),...,\alpha(p_{n})),~~\beta_{V}(f(p_{1},...,p_{n}))= f(\beta(p_{1}),...,\beta(p_{n}))$$
denoted by
$$C^{0}_{BiHom}(g,V)=\{v \in V | \alpha_{V}(v)=v , \beta_{V}(v)=v\}$$and
	$$C^{n} _{BiHom}(g,V)=\{f \in C^{n}(g,V) | \alpha_{V}\circ f=f\circ\alpha \textit{ and } \beta_{V}\circ f=f\circ\beta\}$$
The coboundary operator $\delta_{BiHom}: C^{n}_{BiHom}(g,V )\longrightarrow C^{n+1}_{ BiHom}(g,V )$, for $n \geq 0$ is given by
\begin{enumerate}\item
\begin{equation*} \delta_{BiHom} (p)(v)= \alpha\beta^{-1}(p) \bullet v,~~  for  ~~v \in C^{0}_{BiHom}(g,V )~~ and~~ p \in g.
\end{equation*} \item\begin{eqnarray*}&\delta_{BiHom}f (p_{1},\cdots, p_{n+1}) \\=&\sum_{i=1}^{n+1}(-1)^{i} \alpha\beta^{ n-1}(p_{i})\bullet f(p_{1},\cdots, \hat{p_{i}},\cdots,p_{n+1}) \\&+\sum_{1\leq i<j\leq n+1}(-1)^{i+j+1}f([\alpha^{-1}\beta(p_{i}),p_{j}],\beta(p_{1}),\cdots,\hat{\beta(p_{i})},\cdots,\hat{\beta(p_{j})},\cdots,\beta(p_{n+1}))\end{eqnarray*}
\end{enumerate}for $f \in C^{n}_{BiHom}(g, V )$ and $p_{1},\cdots, p_{n+1} \in g$. The cohomology groups of the cochain complex $$\{C^{\ast}_{BiHom}(g,V ),\delta_{BiHom}\}$$ are called the Chevalley-Eilenberg cohomology groups, denoted by $H^{\ast}_{BiHom}(g, V)$.
It is important to note that the coboundary operator for the Chevalley-Eilenberg cohomology of the
BiHom-Lie algebra $(g, [\cdot,\cdot], \alpha, \beta)$ with coefficients in itself is simply given by 
$$\delta_{BiHom}f = (-1)^{n-1}[\mu, f]_{NR} , ~~\forall ~f \in C^{n}_{BiHom}(g, g),$$
where $\mu \in C^{2}_{BiHom}(g,g)$ corresponds to the BiHom-Lie bracket $[\cdot,\cdot]$.
%%%%%%%%%%%%%%%%%%%%%%%%%%%%%%%%%%%%%%%%%%%%%%%%%%%%%%%%%%%%%%%%%%%%%%%%%%%%%%%%%%%%%%%%%%%%%%%%%%%%%%%%%%
%%%%%%%%%%%%%%%%%%%%%%%%%%%%%%%%%%%%%%%%%%%%%%%%%%%%%%%%%%%%%%%%%%%%%%%%%%%%%%%%%%%%%%%%%%%%%%%%%%%%%%%%%%
%%%%%%%%%%%%%%%%%%%%%%%%%%%%%%%%%%%%%%%%%%%%%%%%%%%%%%%%%%%%%%%%%%%%%%%%%%%%%%%%%%%%%%%%%%%%%%%%%%%%%%%%%%
%%%%%%%%%%%%%%%%%%%%%%%%%%%%%%%%%%%%%%%%%%%%%%%%%%%%%%%%%%%%%%%%%%%%%%%%%%%%%%%%%%%%%%%%%%%%%%%%%%%%%%%%%%
\subsection{Bidifferential graded Lie algebras}
Before presenting the notion of bidifferential graded Lie algebras, let us first give the definition of a differential graded Lie algebra.
\begin{defn}A differential graded Lie algebra is a triple $(L = \oplus L^{i}, [\cdot , \cdot], d)$ consisting of a graded Lie algebra together with a differential $d : L \longrightarrow L$ of degree $+1$ which is a derivation for the bracket $[\cdot ,\cdot ]$. An element $\theta \in L^{1}$ is said to be a Maurer-Cartan element in the differential graded Lie algebra $(L, [\cdot,\cdot ], d)$, if $\theta$ satisfies Maurer-Cartan equation: $$d\theta + \frac{1}{2}[\theta, \theta] = 0.$$
\end{defn}\begin{defn}
A bidifferential graded Lie algebra is a quadruple $(L = \oplus L^{i}, [\cdot , \cdot], d_{1}, d_{2})$ in which the triples $(L, [\cdot , \cdot], d_{1})$ and $(L, [\cdot , \cdot], d_{2})$ are differential graded Lie algebras additionally satisfying $d_{1} \circ d_{2} + d_{2} \circ d_{1} = 0$.
\end{defn}\begin{re}
Any graded Lie algebra can be considered as a bidifferential graded Lie algebra with both
the differentials $d_{1}$ and $d_{2}$ to be trivial.
\end{re}\begin{defn}
Let $(L, [\cdot , \cdot], d_{1}, d_{2})$ be a bidifferential graded Lie algebra. A pair of elements $(\theta_{1}, \theta_{2}) \in L^{1}  \oplus L^{1}$ is said to be a Maurer-Cartan element if
\begin{enumerate}
	\item   $ \theta_{1}$ is a Maurer-Cartan element in the differential graded Lie algebra $(L, [ , ], d_{1});$
\item $\theta_{2}$ is a Maurer-Cartan element in the differential graded Lie algebra $(L, [ , ], d_{2});$
	\item   the following compatibility condition holds $$d_{1}\theta_{2} + d_{2}\theta_{1} + [\theta_{1}, \theta_{2}] = 0.$$
\end{enumerate}Like a differential graded Lie algebra can be twisted by a Maurer-Cartan element, the same result holds for bidifferential graded Lie algebras.
\end{defn}\begin{prop}
Let $(L, [\cdot , \cdot], d_{1}, d_{2})$ be a bidifferential graded Lie algebra and let $(\theta_{1}, \theta_{2})$ be a Maurer-Cartan element.
Then the quadruple $(L, [\cdot , \cdot], d_{1}^{\theta_{1}}, d_{2}^{\theta_{2}})$ is a bidifferential graded Lie algebra, where
$$d_{1}^{\theta_{1}} = d_{1} + [\theta_{1},-]~~and ~~d_{2}^{\theta_{2}} = d_{2} + [\theta_{2},-].$$
For any $v_{1}, v_{2} \in L^{1}$, the pair $(\theta_{1} + v_{1}, \theta_{2} + v_{2})$ is a Maurer-Cartan element in the bidifferential graded Lie algebra $(L, [\cdot , \cdot], d_{1}, d_{2})$ if and only if $(v_{1}, v_{2})$ is a Maurer-Cartan element in the bidifferential graded Lie algebra $(L, [\cdot , \cdot], d_{1}^{\theta_{1}} , d_{2}^{\theta_{2}} )$.
\end{prop}
\section{Compatible BiHom-Lie algebras}
In this section, we introduce compatible Bihom-Lie algebras and give a Maurer-Cartan characterization. We end this section by defining representations of compatible Bihom-Lie algebras. Let $g$ be a vector space and $\alpha , \beta: g \rightarrow g$  be two linear commuting maps.
\begin{defn}
Two BiHom-Lie algebras $(g, [\cdot ,\cdot ]_{1},\alpha,\beta)$ and $(g, [\cdot , \cdot]_{2},\alpha,\beta)$ are said to be compatible if for all
$\lambda,\eta \in \mathbb{K}$, the $4$-tuple $(g,\lambda[\cdot,\cdot]_{1} + \eta[\cdot,\cdot]_{2},\alpha,\beta)$ is a BiHom-Lie algebra.
The compatibility condition in this definition is equivalent to the following:
\begin{eqnarray*}
&[\beta^{2}(p), [\beta(q), \alpha(r)]_{1}]_{2} + [\beta^{2}(q), [\beta(r), \alpha(p)]_{1}]_{2} + [\beta^{2}(r), [\beta(p), \alpha(q)]_{1}]_{2}\\& + [\beta^{2}(p), [\beta(q), \alpha(r)]_{2}]_{1} + [\beta^{2}(q), [\beta(r), \alpha(p)]_{2}]_{1} + [\beta^{2}(r), [\beta(p), \alpha(q)]_{2}]_{1} = 0
\end{eqnarray*} for all $p, q, r \in g$.
\end{defn}\begin{defn}
A compatible BiHom-Lie algebra is a $5$-tuple $(g, [\cdot ,\cdot ]_{1}, [\cdot ,\cdot ]_{2}, \alpha,\beta)$ in which $(g, [\cdot ,\cdot ]_{1},\alpha,\beta)$ and $(g, [\cdot , \cdot]_{2},\alpha,\beta)$ are both BiHom-Lie algebras and are compatible.
In this case, we say that the pair $([\cdot ,\cdot ]_{1}, [\cdot , \cdot]_{2})$ is a compatible BiHom-Lie algebra structure on $g$ when the
twisting maps $\alpha$ and $\beta$ are clear from the context.\\ Compatible BiHom-Lie algebras are twisted version of compatible Lie algebras.
Recall that a compatible Lie algebra is a triple $(g, [\cdot ,\cdot ]_{1}, [\cdot ,\cdot ]_{2})$ in which $(g, [\cdot ,\cdot ]_{1})$ and $(g, [\cdot,\cdot]_{2})$ are Lie algebras and are compatible in the sense that $\lambda[\cdot ,\cdot ]_{1} +\eta[\cdot , \cdot]_{2}$ is a Lie bracket on $g$, for all $\lambda, \eta \in \mathbb{K}$. Thus, a compatible BiHom-Lie algebra $(g, [\cdot ,\cdot ]_{1}, [\cdot ,\cdot ]_{2},\alpha,\beta)$ with $\alpha =  \beta = id$ is nothing but a compatible Lie algebra.
\end{defn}
\begin{defn}
Let $(g, [\cdot ,\cdot ]_{1}, [\cdot ,\cdot ]_{2}, \alpha,\beta)$ and $(g^{\prime}, [\cdot, \cdot ]_{1}^{\prime}
,[\cdot ,\cdot ]_{2}^{\prime}, \alpha^{\prime},\beta^{\prime})$ be two compatible BiHom-Lie algebras. A morphism between them is a linear map $\phi : g\rightarrow g^{\prime}$ which is a BiHom-Lie algebra morphism from $(g, [\cdot,\cdot]_{1}, \alpha,\beta)$
to $(g^{\prime}, [\cdot ,\cdot ]_{1}^{\prime},\alpha^{\prime},\beta^{\prime})$, and a BiHom-Lie algebra morphism from $(g, [\cdot ,\cdot ]_{2}, \alpha,\beta)$
to $(g^{\prime}, {[\cdot ,\cdot ]}_{2}^{\prime}, \alpha^{\prime}, \beta^{\prime}).$
\end{defn} 
\begin{ex} Let $(g, [\cdot , \cdot],\alpha,\beta)$ be a BiHom-Lie algebra. A linear operator $N: g\longrightarrow g$ is called a BiHom-Nijenhuis operator if $[N(p), N(q)]= N([ N(p), q]- [ N(q), p]- N([p, q]))$, for all $p, q \in g$. Then there is a deformed BiHom-Lie Bracket on $g$ given by $$[p, q]_{N}= [N(p), q]- [ N(q), p]- N([p, q])$$
In the other words $(g,[\cdot , \cdot]_{N},\alpha,\beta)$ is a BiHom-Lie algebra.
It is easy to see that the 5-tuple $(g,[\cdot , \cdot],[\cdot , \cdot]_{N},\alpha,\beta)$ is a compatible BiHom-Lie algebra.
\end{ex}
\begin{defn}
Let $(g, [\cdot , \cdot], \alpha,\beta)$ be a BiHom-Lie algebra and $s,l$ be a non-negative integer, $\lambda \in \mathbb{K}$. If A linear operator $R : g \longrightarrow g$ satisfying\begin{align*}
	\alpha \circ R &= R \circ \alpha\\
	\beta \circ R &= R \circ \beta\\
	[R(p),R(q)] &= R([\alpha^{s}\beta^{l}R(p),q] + [p,\alpha^{s}\beta^{l}R(q)] +\lambda[p,q])
	\end{align*} for all $p, q \in g$, then $R$ is called an $sl$-Rota-Baxter operator of weight $\lambda$ on $(g, [\cdot , \cdot], \alpha, \beta)$
\end{defn}
A $sl$-Rota-Baxter operator $R$ induces a new BiHom-Lie algebra structure on $g$ with the BiHom-Lie Bracket
\begin{center}
$[p,q]_{R} = [\alpha^{s}\beta^{l}R(p),q] + [p,\alpha^{s}\beta^{l}R(q)] +\lambda[p,q]$
\end{center}\begin{defn} Two $sl$-Rota-Baxter operators $R$ and $S$ of same $\lambda \in \mathbb{K}$ on a BiHom-Lie algebra $(g, [\cdot , \cdot],\alpha,\beta)$ are said to be compatible if
$$[R(p),S(q)]+[S(p),R(q)]=R([\alpha^{s}\beta^{l}S(p),q]+[p,\alpha^{s}\beta^{l}S(q)])+ S([\alpha^{s}\beta^{l}R(p),q]+[p,\alpha^{s}\beta^{l}R(q)]).$$
\end{defn}
\begin{prop}
Let $R$ and $S$ be two compatible $sl$-Rota-Baxter operators of weight $\lambda \in \mathbb{K}$ on a BiHom-Lie algebra $(g,[\cdot , \cdot],\alpha,\beta)$, Then $(g,[\cdot , \cdot]_{R},[\cdot , \cdot]_{S},\alpha,\beta)$ is a compatible BiHom-Lie algebra.
\end{prop}
Let $g$ be a vector space and $\alpha, \beta : g \longrightarrow g$ be a linear map. Consider the graded Lie algebra $(C^{\ast+ 1}_{BiHom}(g, g), [\cdot , \cdot]_{NR})$ the quadruple $$(C^{\ast+1} _{BiHom}(g, g), [\cdot , \cdot]_{NR}, d_{1} = 0, d_{2} = 0)$$ is a bidifferential graded Lie algebra. Then we have the following Maurer-Cartan characterization of compatible BiHom-Lie algebras.
\begin{thm}
	There is a one-to-one correspondence between compatible BiHom-Lie algebra structures on $g$
	and Maurer-Cartan elements in the bidifferential graded Lie algebra
	$(C^{\ast+1} _{BiHom}(g,g), [\cdot , \cdot]_{NR}, d_{1} = 0, d_{2} = 0)$.
\end{thm}
\begin{proof}
	Let $[\cdot , \cdot]_{1}$ and $[\cdot , \cdot]_{2}$ be two multiplicative skew-symmetric bilinear brackets on $g$. Then the brackets $[\cdot , \cdot]_{1}$ and $[\cdot , \cdot]_{2}$ correspond to elements (say, $\mu_{1}$ and $\mu_{2}$, respectively) in $C^{2}_{BiHom}(g,g)$. Then
	\begin{enumerate}
		\item 	$[\cdot , \cdot]_{1}$ is a BiHom-Lie bracket $\Longleftrightarrow[\mu_{1}, \mu_{1}]_{NR} = 0;$ \\
		\item $[\cdot , \cdot]_{2}$ is a BiHom-Lie bracket $\Longleftrightarrow [\mu_{2}, \mu_{2}]_{NR} = 0;$\\
		\item compatibility   condition (3)$\Longleftrightarrow [\mu_{1}, \mu_{2}]_{NR} = 0.$
	\end{enumerate}Hence $(g, [\cdot , \cdot]_{1}, [\cdot , \cdot]_{2}, \alpha,\beta)$ is a compatible BiHom-Lie algebra if and only if $(\mu_{1}, \mu_{2})$ is a Maurer-Cartan element in the bidifferential graded Lie algebra $(C^{\ast+1}_{BiHom}(g, g), [\cdot , \cdot]_{NR}, d_{1} = 0, d_{2} = 0)$.
\end{proof}\begin{prop}
Let $(g, [\cdot , \cdot]_{1}, [\cdot , \cdot]_{2},\alpha,\beta)$ be a compatible BiHom-Lie algebra. Then for any multiplicative skew-symmetric bilinear operations $[\cdot , \cdot]^{\prime}_{1}$ and $[\cdot , \cdot]^{\prime}_{2}$ on $g$, the 5-tuple
$$(g, [\cdot , \cdot]_{1} + [\cdot , \cdot]^{\prime}_{1}, [\cdot , \cdot]_{2} + [\cdot , \cdot]^{\prime}_{2},\alpha,\beta)$$
is a compatible BiHom-Lie algebra if and only if $(\mu^{\prime}_{1}, \mu^{\prime}_{2})$ is a Maurer-Cartan element in the bidifferential graded Lie algebra $(C^{\ast+1}_{BiHom}(g, g), [\cdot , \cdot]_{NR}, d_{1} = [\mu_{1},-], d_{2} = [\mu_{2},-])$. Here $\mu^{\prime}_{1}, \mu^{\prime}_{2} \in C^{2}_{BiHom}(g, g)$ denote the elements corresponding to the brackets $[\cdot , \cdot]^{\prime}_{1}$  and $[\cdot , \cdot]^{\prime}_{2}$, respectively.
\end{prop}In the following, we define representations of a compatible BiHom-Lie algebra.\begin{defn}
A representation of the compatible BiHom-Lie algebra $(g, [\cdot , \cdot]_{1}, [\cdot , \cdot]_{2}, \alpha,\beta)$ consists of a
5-tuple $(V, \bullet_{1}, \bullet_{2}, \alpha_{V},\beta_{V})$ such that\\
(i)$ (V, \bullet_{1},\alpha_{V},\beta_{V})$ is a representation of the BiHom-Lie algebra $(g,[. , .]_{1},\alpha,\beta)$\\
(ii) $(V, \bullet_{2},\alpha_{V}, \beta_{V})$ is a representation of the BiHom-Lie algebra $(g,[. , .]_{2},\alpha,\beta)$\\
(iii) the following compatibility condition holds\\
$[\beta(p),q]_{1} \bullet_{2} \beta_{V}(v) + [\beta(p), q]_{2} \bullet_{1} \beta_{V}(v) = \alpha\beta(p) \bullet_{1} (q \bullet_{2} v) - \beta(q) \bullet_{2} (\alpha(p) \bullet_{1} v) + \alpha\beta(p) \bullet_{2} (q \bullet_{1} v) - \beta(q) \bullet_{1} (\alpha(p) \bullet_{2} v)$,
for all $p, q \in g$ and $v \in V $.
It follows that any compatible BiHom-Lie algebra $(g, [\cdot , \cdot]_{1}, [\cdot , \cdot]_{2},\alpha,\beta)$ is a representation of itself, where
$\bullet_{1} = [\cdot , \cdot]_{1}$ and $\bullet_{2} = [\cdot , \cdot]_{2}$. This is called the adjoint representation.
\end{defn}
\begin{re}Let $(g, [\cdot , \cdot]_{1}, [\cdot , \cdot]_{2}, \alpha,\beta)$ be a compatible BiHom-Lie algebra and $(V, \bullet_{1}, \bullet_{2},\alpha_{V},\beta_{V})$ be a representation of it. Then for any $\lambda, \eta \in \mathbb{K}$, the 4-tuple $(g, \lambda[\cdot , \cdot]_{1} + \eta[\cdot , \cdot]_{2}, \alpha,\beta)$ is a BiHom-Lie algebra and $(V, \lambda \bullet_{1} +\eta \bullet_{2},\alpha_{V},\beta_{V})$ is a representation of it. The proof of the following proposition is similar to the standard case.
\end{re}\begin{prop}
Let $(g, [\cdot , \cdot]_{1}, [\cdot , \cdot]_{2},\alpha,\beta)$ be a compatible BiHom-Lie algebra and $(V, \bullet_{1}, \bullet_{2},\alpha_{V},\beta_{V})$ be a representation of it. Then the direct sum $g \oplus V$ carries a compatible BiHom-Lie algebra structure with the linear homomorphism $\alpha \oplus \alpha_{V}$ and $\beta\oplus\beta_{V}$, and BiHom-Lie brackets
$$[(p, a), (q, b)]_{i}^{\ltimes} = ([p, q]_{i}, x \bullet_{i} b - (\alpha^{-1}\beta(y)) \bullet_{i}( \alpha_{V}\beta_{V}^{-1}(a)),~~for~~~ i = 1, 2  ~~and~~ (p, a), (q, b) \in g \oplus V.$$ This is called the semidirect product.
\end{prop}\section{Cohomology of compatible BiHom-Lie algebras}
In this section, we introduce the cohomology of a compatible BiHom-Lie algebra with coefficients in a representation.
Let $(g, [\cdot , \cdot]_{1}, [\cdot , \cdot]_{2}, \alpha,\beta)$ be a compatible BiHom-Lie algebra and $(V, \bullet_{1}, \bullet_{2}, \alpha_{V}, \beta_{V})$ be a representation of it. Let $^{1}\delta_{BiHom} : C^{n} _{BiHom}(g, V )\longrightarrow C^{n+1} _{BiHom}(g, V )$ (resp.$^{2}\delta_{BiHom}: C^{n} _{BiHom}(g, V )\longrightarrow C^{n+1} _{BiHom}(g, V )$, for $n \geq 0$, be the coboundary operator for the Chevalley-Eilenberg cohomology of the BiHom-Lie algebra $(g, [\cdot, \cdot]_{1},\alpha,\beta)$ with coefficients in the representation $(V, \bullet_{1},\alpha_{V},\beta_{V})$ (resp. of the BiHom-Lie algebra $(g, [\cdot , \cdot]_{2},\alpha,\beta)$ with coefficients in the representation $(V, \bullet_{2},\alpha_{V},\beta_{V})$. Then we have\\
\begin{center}
$(^{1}\delta_{BiHom})^{2} = 0$ and  $(^{2}\delta_{BiHom})^{2} = 0$.
\end{center}
Moreover, we have the following.\begin{prop}\label{prop4.1} \begin{center} The coboundary operators $^{1}\delta_{BiHom}$ and $^{2} \delta_{BiHom}$ satisfy the following compatibility\\ $^{1}\delta_{BiHom} \circ   ^{2}\delta_{BiHom} + ^{2}\delta_{BiHom} \circ   ^{1}\delta_{BiHom} = 0.$ \end{center} \end{prop} Before we prove the above proposition, we first observe the followings. For a compatible BiHom-Lie algebra
$(g, [\cdot , \cdot]_{1}, [\cdot , \cdot]_{2},\alpha,\beta)$ and a representation $(V,\bullet_{1}, \bullet_{2},\alpha_{V},\beta_{V})$, we consider the semidirect product compatible BiHom-Lie algebra structure on $g\oplus V$ given in Proposition 3.6. We denote by\\
$\pi_{1}, \pi_{2} \in C^{2}_{BiHom}(g\oplus V, g\oplus V )$ the elements corresponding to the BiHom-Lie brackets $[\cdot , \cdot]^{\ltimes}_{1}$ and $[\cdot , \cdot]^{\ltimes}_{2}$ on $g \oplus V$ , respectively. Let
\begin{center}

$^{1}\delta_{BiHom} : C^{n} _{BiHom}(g \oplus V, g \oplus V ) \longrightarrow C^{n+1} _{BiHom}(g \oplus V, g \oplus V )$, for $n \geq0$,\\
$^{2}\delta_{BiHom} : C^{n} _{BiHom}(g \oplus V, g \oplus V ) \longrightarrow C^{n+1} _{BiHom}(g \oplus V, g \oplus V )$, for $n \geq 0$,
\end{center}denote respectively the coboundary operator for the Chevalley-Eilenberg cohomology of the BiHom-Lie algebra
$(g \oplus V, [\cdot , \cdot]^{\ltimes}_{1} , \alpha\oplus \alpha_{V}, \beta\oplus \beta_{V})$ (resp. of the BiHom-Lie algebra $(g \oplus V, [\cdot , \cdot]^{\ltimes}_{2} , \alpha\oplus \alpha_{V}, \beta\oplus \beta_{V})$ with coefficients in itself. Note that any map $f \in C^{n} _{BiHom}(g \oplus V)$ can be lifted to a map $\tilde{f} \in C_{BiHom}^{n}(g\oplus V,g \oplus V)$ by
\begin{center}
$\tilde{f}((p_{1},v_{1}),...,(p_{n},v_{n}))=(0, f(p_{1},...,p_{n})).$
\end{center}
Then $f = 0$ if and only if $\tilde{f}$ = 0.\\
With these notations, for any $f \in  C^{n} _{BiHom}(g,V)$, we have\\
$(\widetilde{^{1}\delta_{BiHom}f})=\delta^{1}_{BiHom}(\tilde{f})=(-1)^{n-1}[\pi_{1},\tilde{f}]_{NR}$ and $(\widetilde{^{2}\delta_{BiHom}f})=\delta^{2}_{BiHom}(\tilde{f})=(-1)^{n-1}[\pi_{2},\tilde{f}]_{NR}$.\begin{proof}(Proposition \ref{prop4.1}) For any $f \in C^{n}_{BiHom}(g,V)$, we have\begin{align*}
	\widetilde{(^{1}\delta_{BiHom}\circ^{2}\delta_{BiHom}+{^{2}\delta_{BiHom}}\circ^{1}\delta_{BiHom})(f)}&=\widetilde{^{1}\delta_{BiHom}(^{2}\delta_{BiHom}f)}+\widetilde{^{2}\delta_{BiHom}(^{1}\delta_{BiHom}f)}\\
	&=(-1)^{n}[\pi_{1},\widetilde{^{2}\delta_{BiHom}f}]_{NR}+(-1)^{n}[\pi_{2},\widetilde{^{1}\delta_{BiHom}f}]_{NR}\\
	&=-[\pi_{1}, [\pi_{2},\tilde{f}]_{NR}]_{NR} - [\pi_{2}, [\pi_{1},\tilde{f}]_{NR}]_{NR}\\
	&= [\pi_{2}, [\pi_{1},\tilde{f}]_{NR}]_{NR}-[[\pi_{1}, \pi_{2}]_{NR}, \tilde{f}]_{NR}  - [\pi_{2}, [\pi_{1},\tilde{f}]_{NR}]_{NR}\\
	&=0\end{align*} \begin{center}
 where $([\pi_{1},\pi_{2}]_{NR} = 0)$.\end{center}Therefore, it follows that $(^{1}\delta_{BiHom}\circ ^{2}\delta_{BiHom}+^{2}\delta_{BiHom}\circ ^{1}\delta_{BiHom})(f)= 0$. Hence the result follows.\end{proof}We are now in a position to define the cohomology of a compatible BiHom-Lie algebra\\$(g, [\cdot, \cdot]_{1}, [\cdot, \cdot]_{2},\alpha,\beta)$ with coefficients in a representation $(V, \bullet_{1}, \bullet_{2},\alpha_{V},\beta_{V})$. For each $n \geq 0$, we define an abelian group $C^{n}_{cBiHom}(g, V )$ as follows:
\begin{center}
\begin{align*}
C^{0}_{cBiHom}(g, V) &= C^{0}_{BiHom}(g,V) \cap \{v \in V | \alpha\beta^{-1}(p) \bullet_{1} v = \alpha\beta^{-1}(p)\bullet_{2} v, \forall p \in g\}\\
&= \{v \in V | \alpha_{V}(v) = v\textit{ , } \beta_{V}(v) = v \textit{ and } \alpha\beta^{-1}(p) \bullet_{1} v = \alpha\beta^{-1} (p) \bullet_{2} v, \forall p \in g\},
\end{align*}
$C^{n}_{cBiHom}(g, V ) = \underbrace{C^{n}_{BiHom}(g, V )\oplus...\oplus C^{n}_{BiHom}(g, V )}_{\textit{n- copies}}$, for $n \geq 1$.
\end{center}Define a map $\delta_{cBiHom }: C^{n}_{cBiHom}(g, V ) \longrightarrow C^{n+1}_{cBiHom}(g,V ), for n \geq 0$ by\begin{center}
$\delta_{cBiHom}(v)(p) = \alpha\beta^{-1}(p) \bullet_{1} v = \alpha\beta^{-1}(p) \bullet_{2} v$, for $v \in C^{0}_{ cBiHom}(g, V )$ and $p \in g,$\\
$\delta_{cBiHom}(f_{1},...,f_{n}) =(^{1}\delta_{BiHom}f_{1},...,\underbrace{^{1}\delta_{BiHom}f_{i} + ^{2}\delta_{BiHom}f_{i-1}}_{\textit{i-th position}},...,^{2}\delta_{BiHom}f_{n})$,
\end{center}
for $(f_{1},...,f_{n}) \in C^{n}_{cBiHom}(g, V )$.Then we have the following.
\begin{prop}
	The map $\delta_{cBiHom}$ is a coboundary map, i.e., $(\delta_{cBiHom})^{2} = 0$.
\end{prop}\begin{proof}
For any $v \in C^{0}_{cBiHom}(g,V)$, we have
\begin{center}
	$(\delta_{cBiHom})^{2}(v) = \delta_{cBiHom}(\delta_{cBiHom}v) = (^{1}\delta_{BiHom}\delta_{cBiHom}v , ^{2}\delta_{BiHom}\delta_{cBiHom}v)$\\
	$=(^{1}\delta_{BiHom}$ $^{1}\delta_{BiHom}v $,$^{ 2}\delta_{BiHom}$  $^{2}\delta_{BiHom}v) = 0.$
	\end{center}Moreover, for any $(f_{1},...,f_{n})\in C^{n}_{cBiHom}(g,V), n \geq 1$, we have \begin{center}\begin{align*}&(\delta_{cBiHom})^{2}(f_{1},...,f_{n})\\&= \delta_{cBiHom}( ^{1} \delta_{BiHom}f_{1},...,^{1}\delta_{BiHom}f_{i} + ^{2}\delta_{BiHom}f_{i-1},...,^{2}\delta_{BiHom}f_{n})\\ &=( {^{1}\delta_{BiHom}} {^{1}\delta_{BiHom}}f_{1} , {^{2}\delta_{BiHom}}{^{1}\delta_{BiHom}}f_{1}  + {^{1}\delta_{BiHom}} {^{2}\delta_{BiHom}}f_{1}  + {^{1}\delta_{BiHom}} {^{1}\delta_{BiHom}f_{2}} ,...,\\&\underbrace{{^{2}\delta_{BiHom}} {^{2}\delta_{BiHom}}f_{i-2}  +  {^{2}\delta_{BiHom}}   {^{1}\delta_{BiHom}}f_{i-1}  +  {^{1}\delta_{BiHom}} {^{2}\delta_{BiHom}}f_{i-1} +  {^{1}\delta_{BiHom}}  {^{1}\delta_{BiHom}f_{i}} ,...,}_{3\leq i \leq n-1}\\&{^{2}\delta_{BiHom}} {^{2}\delta_{BiHom}}f_{n-1}  +  {^{2}\delta_{BiHom}} {^{1}\delta_{BiHom}}f_{n}  +  {^{1}\delta_{BiHom}} {^{2}\delta_{BiHom}}f_{n}, { ^{2}\delta_{BiHom}} {^{2}\delta_{BiHom}}f_{n} \\&= 0\end{align*}\end{center}
This proves that $(\delta_{cBiHom})^{2} = 0$.
	\end{proof}It follows from the above proposition that $\{C^{\ast}_{cBiHom}(g,V), \delta_{cBiHom}\}$ is a cochain complex. The corresponding cohomology groups\begin{center}
	$H^{n}_{cBiHom}(g, V ) =\frac{ Z^{n}_{cBiHom}(g,V)}{B^{n}_{cBiHom}(g,V)}= \frac{ Ker    \delta_{cBiHom} : C^{n}_{cBiHom}(g,V) \longrightarrow C^{n+1}_{cBiHom} (g,V)}{   Im    \delta_{cBiHom} : C^{n-1}_{cBiHom}(g,V) \longrightarrow C^{n}_{cBiHom}(g,V)}$ , for $n \geq 0$
\end{center}are called the cohomology of the compatible BiHom-Lie algebra $(g,[. , .]_{1}, [. , .]_{2},\alpha,\beta)$ with coefficients in the representation $(V,\bullet_{1},\bullet_{2},\alpha_{V},\beta_{V})$.

\par Let $g = (g, [\cdot , \cdot]_1, [\cdot , \cdot]_2, \a,\b)$ be a compatible BiHom-Lie algebra and $V = (V, \bullet_{1}, \bullet_{2}, \a_V,\b_V)$ be a representation of it.
Then we know from Remark 3.15 that $g_+ = (g, [\cdot , \cdot]_1 + [\cdot , \cdot]_2, \a+\b)$ is a BiHom-Lie algebra and $V_+ = (V, \bullet_{1}+ \bullet_{2} ,\a_V+\b_V)$ is a representation of it. Consider the cochain complex $\{C^*_{cBiHom}(g, V), \delta_{BiHom}\}$ of the compatible BiHom-Lie algebra $g$ with coefficients in the representation $V$ , and the cochain complex $\{C^*_{BiHom}(g_+, V_+),\delta_{BiHom}\}$ of the BiHom-Lie algebra $g_+$ with coefficients in $V$. Like wise in \cite{1}, we can establish a following theorem:
\begin{thm} The collection $\{\varphi\}_{n\geq0}$ defines a morphism of cochain complexes from to $\{C^*_{BiHom}(g_+, V_+), \delta_{BiHom}\}$. Hence, it induces a morphism $H^*_{cBiHom}(g, V ) \to H^{*}_{BiHom} (g_{+}, V_+)$ between corresponding cohomologies. where $\varphi_n$ is a map $\varphi_{n\geq0} : C^n_{cBiHom}(g,V) \to C^n_{BiHom}(g_+, V_+)$ defined by\begin{eqnarray}\varphi_{n\geq0}=\biggl\{\varphi_{0}(v)= &\frac12(v),~~ v \in C^0_{cBiHom}(g,V) \\\varphi_{n\geq1}((f_1, \cdots, f_n))=& f_1+ \cdots+ f_n, ~~ (f_1, \cdots, f_n)\in C^{n\geq1}_{cBiHom}(g,V)\biggr\}
	\end{eqnarray}
\end{thm}
\noindent {\bf Acknowledgment:}
The authors would like to thank the referee for valuable comments and suggestions on this article.

\begin{thebibliography}{999}
\bibitem{1}Ammar, F., Ejbehi, Z.,  Makhlouf, A. (2010). Cohomology and deformations of Hom-algebras. arXiv preprint arXiv:1005.0456.
\bibitem{2}Cheng, Y.,  Qi, H. (2022, March). Representations of BiHom-Lie algebras. In Algebra Colloquium (Vol. 29, No. 01, pp. 125-142). World Scientific Publishing Company.
\bibitem{3} Das, A. (2022). Cohomology and deformations of compatible Hom-Lie algebras. arXiv preprint arXiv:2202.03137.
\bibitem{4}Das, A.,  Sen, S. (2022). Nijenhuis operators on Hom-Lie algebras. Communications in Algebra, 50(3), 1038-1054.
\bibitem{5} Fan, Y., Zhu, J.,  Chen, L. $\mathcal{O}$-Operators on Bihom-Lie algebras.
\bibitem{6}Graziani, G., Makhlouf, A., Menini, C.,  Panaite, F. (2015). BiHom-associative algebras, BiHom-Lie algebras and BiHom-bialgebras. SIGMA. Symmetry, Integrability and Geometry: Methods and Applications, 11, 086.
\bibitem{7}Liu, L., Makhlouf, A., Menini, C.,  Panaite, F. (2017). BiHom-pre-Lie algebras, BiHom-Leibniz algebras and Rota-Baxter operators on BiHom-Lie algebras. arXiv preprint arXiv:1706.00474.
\bibitem{8}Liu, L., Makhlouf, A., Menini, C.,  Panaite, F. (2022). BiHom-NS-algebras twisted Rota-Baxter operators and generalized Nijenhuis operators. arXiv preprint arXiv:2206.04418.
\bibitem{9} Liu, J., Sheng, Y.,  Bai, C. (2023). Maurer-Cartan characterizations and cohomologies of compatible Lie algebras. Science China Mathematics, 1-22.
\bibitem{10} Mishra, S. K., Naolekar, A. (2020). $\mathcal{O}$-operators on hom-Lie algebras. Journal of Mathematical Physics, 61(12), 121701.
\end{thebibliography}
\end{document}